\documentclass[12pt]{amsart}
\usepackage[cp1250]{inputenc}
\usepackage{graphicx, lpic}
\usepackage{amssymb, amsmath}

\theoremstyle{plain}
\newtheorem{theorem}{Theorem}
\newtheorem{proposition}[theorem]{Proposition}

\newtheorem{exercise}[theorem]{Exercise}

\theoremstyle{definition}
\newtheorem{definition}[theorem]{Definition}
\newtheorem{example}[theorem]{Example}
\newtheorem{remark}[theorem]{Remark}
\newtheorem{question}[theorem]{Question}

\allowdisplaybreaks

\begin{document}

\title[Minimal hard surface-unlink and classical unlink diagrams]{MINIMAL HARD SURFACE-UNLINK\\ AND CLASSICAL UNLINK DIAGRAMS}

\author[M. Jab{\l}onowski]{Micha{\l} Jab{\l}onowski}

\dedicatory{Dedicated to Witold Rosicki on his 65th birthday.}

\address{Institute of Mathematics, Faculty of Mathematics, Physics and Informatics, University of Gda\'nsk, 80-308 Gda\'nsk, Poland}

\email{michal.jablonowski@gmail.com}



\date{July 28, 2018.}

\begin{abstract}
We describe a method for generating minimal hard prime surface-link diagrams. We extend the known examples of minimal hard prime classical unknot and unlink diagrams up to three components and generate figures of all minimal hard prime surface-unknot and surface-unlink diagrams with prime base surface components up to ten crossings.
\end{abstract}

\maketitle

\section{Introduction}
Hard unknots and unlinks are diagrams of trivial knots and links that have to be made more complicated before they can be simplified. In the classical knot theory of circles in $3$-space with their generic diagrams on a $2$-sphere, they have quite a long history reaching the Goeritz example in \cite{Goe34} (which is the mirror of $11_{\{0,34\}}^{\{1,2,Ori\}}$ in our notation). Recent studies of hard unknots are related to the study of a recombination of DNA (see, e.g. \cite{KauLam06}, \cite{KauLam07}), they also are subject of testing the sharpness of new upper bounds on the number of Reidemeister moves needed to unknot an unknot (see, e.g. \cite{Kau16}).

One dimensional higher analogue of classical knots and links can be studied by the marked graph diagrams. These diagrams are planar (or spherical) diagrams obtained by cutting, in the level of all its saddle points, an embedded surface in $\mathbb{R}^4$ in the hyperbolic splitting position. In research papers there are at least two further ways to transform marked graph diagrams: to the banded links, e.g. \cite{Jab16}, \cite{Swe01} and to monoidal structures, e.g. \cite{Jab17}, \cite{KeaKur08}.

It is believed that the results in \cite{Swe01} prove the Yoshikawa conjecture about the complete generating set of local moves between marked graph diagrams presenting surface-links of the same type (this result we use in the next section) and that \cite{KeaKur08} shares more details on this proof.

One has to be careful reading the latter paper. On page 1225, the last two types of relations in $(1-4)$ of the construction of the semi-group do not correspond to valid moves for surface-links, because they change the marked vertex type. One can see that changing for example one vertex type in the standard torus changes its type to the unknotted sphere. Moreover, the top right image in the figure 5 on page 1231 does not represent the spun $2$-knot of the trefoil because it has the trivial band (for a flat banded link representation of the spun $2$-knot of the trefoil see \cite{Jab16}).

As for hard unknots and unlinks from the classical knot theory, it is stated in \cite{JabSaz07} on page 49 that the smallest hard unlink diagrams have $9$ crossings, that there are two diagrams for hard unknots with $9$ crossings and that there are six diagrams for hard unknots and five diagrams for hard unlinks with $10$ crossings. 

This does not meet our computation results because an implementation of our algorithm (described in this paper) generates (up to mirror images) one hard prime unlink with $8$ crossings presented in Fig.\;\ref{hard08}, generates four hard prime unknots with $9$ crossings presented in Fig.\;\ref{hard09} (three of them with pairwise different the number of $2$-gons) and generates more hard prime unlinks and unknots with $10$ crossings (e.g. $10_{\{0,21\}}^{\{1,2,Ori\}}$, $10_{\{0,3\}}^{\{2,4,Ori\}}$ with different the number of $6$-gons and $4$-gons respectively).

In this paper we describe a method for generating hard prime surface-knot and surface-link diagrams and give the EPD codes of the results for their ch-diagrams up to $10$ crossings (and some $11$ and $12$ crossing diagrams). We also extend the tables of minimal hard prime classical unknots and unlinks mentioned above and generate all hard prime classical unknots and unlinks up to $12$ crossings. We also provide hard alternative diagrams to non-hard diagrams from Yoshikawa's table that have the same the number of classical and marked crossings.

For the generation of the sphere partition graphs we used the \emph{plantri} program by G. Brinkmann and B.D. McKay. The Alexander ideals for surface-links were computed using the \emph{KNOT} program by K. Kodama. Images of hard marked graph diagrams were obtained with the help of E. Redelmeier's program \emph{DrawPD}, further modified with the graphical program {\it Inkscape} and are drawn up to mirror image. Our EPD codes were generated using the {\it Mathematica} program and they can be found in the arXiv source file of this article's preprint version. From each such code one can generate unambiguously its spherical diagram.

\section{Basic definitions and theorems}
An embedding (or its image) of a closed (i.e. compact, without boundary) surface $F$ into $\mathbb{R}^4$ is called a \emph{surface-link}, with $F$ being the \emph{based surface} for this embedding. Two surface-links are \emph{equivalent} (or have the same \emph{type}) if there exists an orientation preserving homeomorphism of the four-space $\mathbb{R}^4$ to itself (or equivalently auto-homeomorphism of the four-sphere $\mathbb{S}^4$), mapping one of those surfaces onto the other.

We will work in the standard smooth category and will use a word \emph{classical}, thinking about theory of embeddings of circles $S^1\sqcup\ldots\sqcup S^1\hookrightarrow \mathbb{R}^3$ modulo ambient isotopy in $\mathbb{R}^3$ with their planar or spherical generic projections. When the base surface is the $2$-sphere or its split unions, then we call the surface-knot a \emph{$2$-knot} or \emph{$2$-link} respectively. In this paper the \emph{primeness} of a diagram is considered with respect to connected sum on the sphere and the \emph{minimality} of a diagram is considered with respect to the total number of its crossings (classical and marked).

To describe a knotted surface in $\mathbb{R}^4$, we will use transverse cross-sections $\mathbb{R}^3\times\{t\}\subset\mathbb{R}^4$ for $t\in\mathbb{R}$, denoted by $\mathbb{R}^3_t$. This method introduced by Fox and Milnor was presented in \cite{Fox62}. Let us present a hyperbolic splitting of a surface-link.

\begin{theorem}[\cite{Lom81}, \cite{KSS82}, \cite{Kam89}]
For any surface-link $F$, there exists a surface-link $F'$ satisfying the following: $F'$ is equivalent to $F$ and has only finitely many Morse's critical points, all maximal points of $F'$ lie in $\mathbb{R}^3_1$, all minimal points of $F'$ lie in $\mathbb{R}^3_{-1}$, all saddle points of $F'$ lie in $\mathbb{R}^3_0$.
\end{theorem}

The zero section $\mathbb{R}^3_0\cap F'$ of the surface $F'$ in the \emph{hyperbolic splitting} described above gives us then a $4$-regular graph. We assign to each vertex a \emph{marker} that informs us about one of the two possible types of saddle points (see Fig.\;\ref{pic001}) depending on the shape of the section $\mathbb{R}^3_{-\epsilon}\cap F'$ or $\mathbb{R}^3_{\epsilon}\cap F'$ for a small real number $\epsilon>0$. The resulting (rigid-vertex) graph is called a \emph{marked graph} presenting $F$ (also known as a \emph{ch-diagram}).

\begin{figure}[ht]
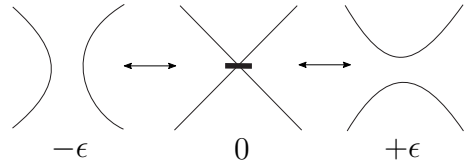

\begin{center}
\begin{lpic}[b(0.5cm)]{./PICTURES/m001(6cm)}
	\lbl[t]{15,-2;$-\epsilon$}
  \lbl[t]{60,-2;$0$}
  \lbl[t]{102,-2;$+\epsilon$}
	\end{lpic}
		\caption{Rules for smoothing a marker.\label{pic001}}
\end{center}
\end{figure}

Making a projection in general position of this graph to $\mathbb{R}^2\times\{0\}\times\{0\}\subset\mathbb{R}^4$ and assigning types of classical crossings between regular arcs, we obtain a \emph{marked graph diagram}. For a marked graph diagram $D$, we denote by $L_+(D)$ and $L_-(D)$ the classical link diagrams obtained from $D$ by smoothing every vertex as presented in Fig.\;\ref{pic001} for $+\epsilon$ and $-\epsilon$ case respectively. We call $L_+(D)$ and $L_-(D)$ the \emph{positive resolution} and the \emph{negative resolution} of $D$, respectively.

Any abstractly created marked graph diagram is a ch-diagram (or it is \emph{admissible}) if and only if both its resolutions are trivial classical link diagrams. In \cite{Yos94} Yoshikawa introduced local moves on admissible marked graph diagrams that do not change corresponding surface-link types and conjectured that the converse is also true. It was resolved as follows.

\begin{theorem}[\cite{Swe01}, \cite{KeaKur08}]
Any two marked graph diagrams representing the same type of surface-link are related by a finite sequence of Yoshikawa local moves presented in Fig.\;\ref{pic002} and their mirror moves (and an isotopy of the diagram in $\mathbb{R}^2$).
\end{theorem}

\begin{figure}[ht]
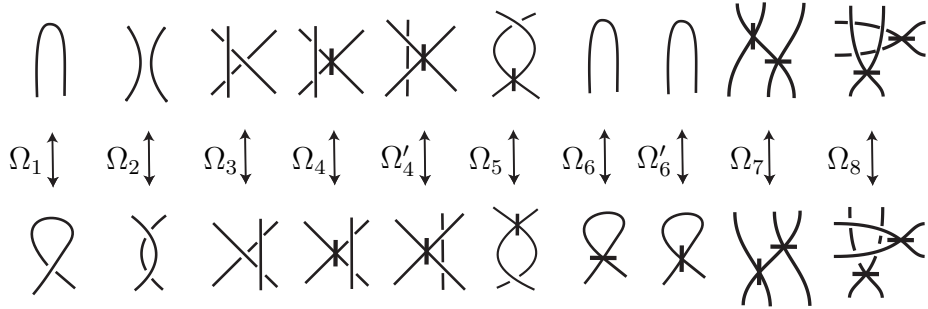

\begin{center}
\begin{lpic}[l(0.4cm)]{./PICTURES/m002(11.9cm)}
	\lbl[r]{3,33;$\Omega_1$}
  \lbl[r]{25,33;$\Omega_2$}
	\lbl[r]{47,33;$\Omega_3$}
	\lbl[r]{67,33;$\Omega_4$}
	\lbl[r]{87,33;$\Omega_4'$}
	\lbl[r]{107,33;$\Omega_5$}
	\lbl[r]{128,33;$\Omega_6$}
	\lbl[r]{145,33;$\Omega_6'$}
	\lbl[r]{166,33;$\Omega_7$}
	\lbl[r]{188,33;$\Omega_8$}
	\end{lpic}
\caption{A set of Yoshikawa moves.\label{pic002}}
\end{center}
\end{figure}

Let us recall the following dependencies between Yoshikawa type moves (from \cite{JKL15}). Where by a \emph{mirror move} we mean the mirror move with respect to classical crossings and by a \emph{switch move} we mean the mirror move with respect to marked vertices.

\begin{itemize}
\item the mirror move to the move $\Omega_1$ can be obtained by the moves $\Omega_1, \Omega_2$ and planar isotopy,
\item the mirror move to the move $\Omega_2$ can be obtained by the move $\Omega_2$ and planar isotopy,
\item the mirror move to the move $\Omega_3$ can be obtained by the moves $\Omega_2, \Omega_3$ and planar isotopy,
\item the switch move to the move $\Omega_4$ can be obtained by the moves $\Omega_2, \Omega_4$ and planar isotopy,
\item the switch move to the move $\Omega_4'$ can be obtained by the moves $\Omega_2, \Omega_4'$ and planar isotopy,
\item the mirror move to the move $\Omega_5$ can be obtained by the moves $\Omega_2, \Omega_5$ and planar isotopy,
\item the switch move to the move $\Omega_5$ can be obtained by the moves $\Omega_1$, $\Omega_2$, $\Omega_3$, $\Omega_4$, $\Omega_4'$, $\Omega_5$ and planar isotopy,
\item the switch and the mirror move to the move $\Omega_5$ can be obtained by the moves $\Omega_1$, $\Omega_2$, $\Omega_3$, $\Omega_4$, $\Omega_4'$, $\Omega_5$ and planar isotopy,
\item the switch move to the move $\Omega_8$ can be obtained by the moves $\Omega_4, \Omega_4',\Omega_8$ and planar isotopy.
\end{itemize}

\section{Generation of ch-diagrams}

\noindent{\bf From sphere partition graphs to shadows}.

We start with generating the isomorphism classes (with respect to the embeddings) of spherical graphs (i.e. graphs that can be drawn on the sphere) that produces partitions of the $2$-sphere. We specifically chose the (simple graph) partitions into regions with exactly four edges at the boundary of each region, which are at least $2$-connected and which vertices are of degree at least two. Picking one graph from each class and excluding mirror images we obtain as a result a set $S_1$.

\begin{figure}[ht]
\begin{center}
\includegraphics[width=0.495\textwidth]{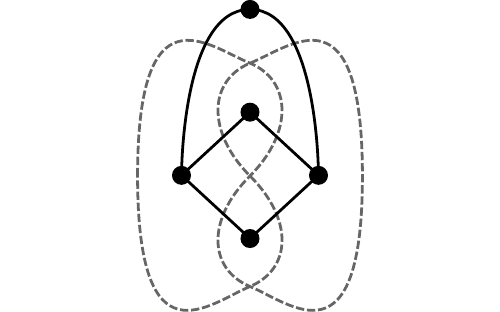}
\includegraphics[width=0.495\textwidth]{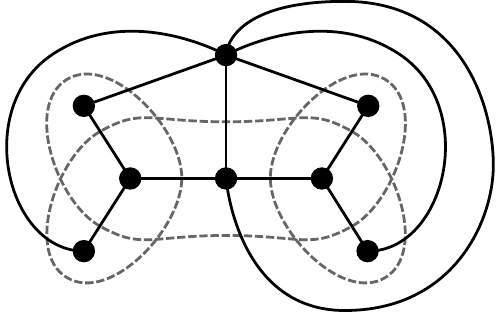}
\end{center}
		\caption{Examples of spherical partitions and their dual graphs, which are link shadows\label{pic224}}
\end{figure}

We then create the dual graph of each graph from the set $S_1$. Therefore we obtain the set of shadows of connected diagrams of classical links, without mirror images, without kinks, without loops, without nugatory crossings, without connected sums of other diagrams (elaboration of various relationships between shadows and graph types one can find in \cite{CCM16} from where we take Fig.\;\ref{pic224}). As a result we obtain here a set $S_2$.

\noindent{\bf From shadows to PD codes of diagrams.}

From each element (shadow) from the set $S_2$ with edges numbered by a pair of letters (that label its ending vertices) we first re-label edges to positive integers, with the caution that we separate the pair of letters that appear four times in the graph to appropriate two pairs of distinct numbers not appeared elsewhere. So that each number is counted exactly twice and is the label of exactly one edge. Then, we re-numerate (for our numeric convenience to have numbers, instead of possible duplication of the letters) edge labels with distinct consecutive positive integers from $1$ up to $2\cdot(\text{number of crossings})$, with increasing labels as we go around each transverse circular component in the shadow (e.g. presented in Fig.\;\ref{M_1_C_10_K_1} when we replace all crossings to the flat ones). We obtain here a set $S_3$.

The number in column $S$ and row $n$ of Table\;\ref{table1} counts the number of all shadows of connected diagrams with $n$ crossings of classical links, without mirror images, without kinks, without loops, without nugatory crossings, without connected sums of other diagrams.

\begin{table}[ht]
\caption{Numerical results of spherical shadows and marked vertex diagrams.\label{table1}}
{\begin{tabular}{@{}lll@{}} 

	n&S&DG\\
	\hline
	2&1&6\\
	3&1&20\\
	4&2&144\\
	5&3&816\\
	6&9&9.504\\
	7&18&74.880\\
		8&62&1.023.744\\
		9&198&13.026.816\\
		10&803&210.912.768\\
		11&3.378&3.545.548.800\\
		12&15.882&66.646.462.464\\
\end{tabular}}
\end{table}

From each element from the set $S_3$ with $n$ flat crossing we generate (by changing every flat crossing to an arbitrary classical crossing) one classical link diagram and store it as a \emph{PD code} (a Planar Diagram code), an unordered set consisting of $n$ elements, each of the form $X[a,b,c,d]$ represents a classical crossing between the edges labelled $a, b, c$ and $d$ starting from the incoming lower strand $a$ and going counterclockwise through $b$, $c$ and $d$. The convention is presented on the left of Fig.\;\ref{pic003}.

\begin{figure}[ht]
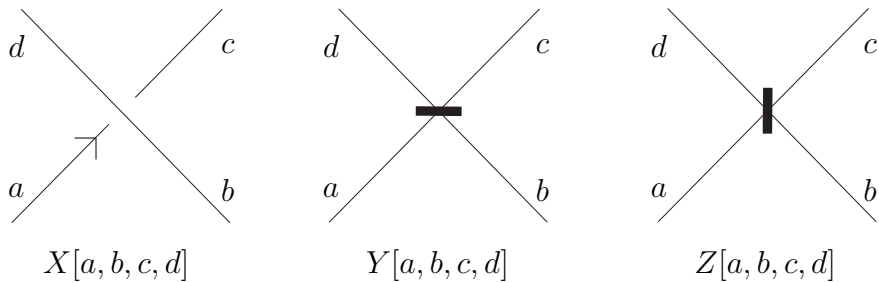

\begin{center}
\begin{lpic}[b(0.5cm),l(0.2cm),r(0.2cm)]{./PICTURES/m003(11.5cm)}
	\lbl[t]{16,-4;$X[a,b,c,d]$}
  \lbl[t]{65,-4;$Y[a,b,c,d]$}
  \lbl[t]{115,-4;$Z[a,b,c,d]$}
	\lbl[r]{2,5;$a$}
	\lbl[r]{50,5;$a$}
	\lbl[r]{100,5;$a$}
	\lbl[r]{2,27;$d$}
	\lbl[r]{50,27;$d$}
	\lbl[r]{100,27;$d$}
	\lbl[l]{32,27;$c$}
	\lbl[l]{80,27;$c$}
	\lbl[l]{130,27;$c$}
		\lbl[l]{32,5;$b$}
	\lbl[l]{80,5;$b$}
	\lbl[l]{130,5;$b$}
	\end{lpic}
		\caption{Interpretation of elements of EPD codes.\label{pic003}}
\end{center}
\end{figure}

From each such PD code, we then generate $2^{n-1}$ possible PD codes by choosing every subset of crossing to be changed by its type (the lower strand over the upper strand) and excluding mirror images. Notice that in the process of changing the type of a crossing by rotating the elements by one in the code for the crossing, the global orientation of component cycles may not stay coherent with the monotonicity of the edge labels (e.g. the $\mathbb{S}^2$-component in Example \ref{exe1}) and one may want to relabel the edges again.

We don't need this orientation any longer, moreover, to further test the orientability of the base surface of a marked graph diagram we will use different orientation.

\noindent{\bf From PD codes of diagrams to EPD codes of marked diagrams.}

We enhance a PD code to an \emph{EPD code} in which the letter $X$ may be replaced either by $Y$ or $Z$ with the convention presented on the right of Fig.\;\ref{pic003}. The number in column $DG$ and row $n$ of Table\;\ref{table1} counts the number of all created from the set $S$ marked diagrams (with $n$ crossings) without mirror images (with respect to classical and marked crossings) in amounts obtained by the equality $DG=S\cdot(4^{n-1}+2^{n-1})$.

We then select as a new set $S_4$ only those marked graphs diagrams from the set $S_3$ with trivial both the positive resolution and the negative resolution (i.e. ch-diagrams). We test the triviality by using the Jones polynomial, which is a valid test for our purpose (i.e. for classical knots and links up to $12$ crossings, see \cite{DasHou97} and \cite{Thi01}). Each EPD code corresponds to the unique surface-link type and every type of a surface-link is represented by an EPD code in our results of the enumeration (as shares the same shadow as some classical diagram).

We will call a \emph{code-crossing} the quadruple from an EPD code with the appropriate letter $X, Y$ or $Z$ as a \emph{name} of this code-crossing, the integers from an EPD code we will call \emph{code-vertices}.

\section{Hard unknots and unlinks}

We start from a definition of surface-unlinks and their standard marked graph diagrams.

\begin{definition}
An orientable surface-link in $\mathbb{R}^4$ is \emph{unknotted} if it is equivalent to a surface embedded in $\mathbb{R}^3\times\{0\}\subset\mathbb{R}^4$. A marked graph diagram for an unknotted \emph{standard sphere} is shown in Fig.\;\ref{m005}(a), an unknotted \emph{standard torus} is in Fig.\;\ref{m005}(d). An embedded projective plane $\mathbb{P}^2$ in $\mathbb{R}^4$ is \emph{unknotted} if it is equivalent to a surface whose marked graph diagram is an unknotted \emph{standard projective plane}, which looks like in Fig.\;\ref{m005}(b) that is a \emph{positive} $\mathbb{P}^2_+$ or looks like in Fig.\;\ref{m005}(c) that is a \emph{negative} $\mathbb{P}^2_-$. These are inequivalent surfaces and are mirror images to one another. A non-orientable surface embedded in $\mathbb{R}^4$ is \emph{unknotted} if it is equivalent to some finite connected sum or split unions of unknotted projective planes.

By a \emph{standard surface-unlink marked diagram} we mean the diagram obtained by taking connected sums or split unions (in the plane) of the standard sphere or the standard tori or the standard projective plane diagrams.
\end{definition}

\begin{figure}[ht]
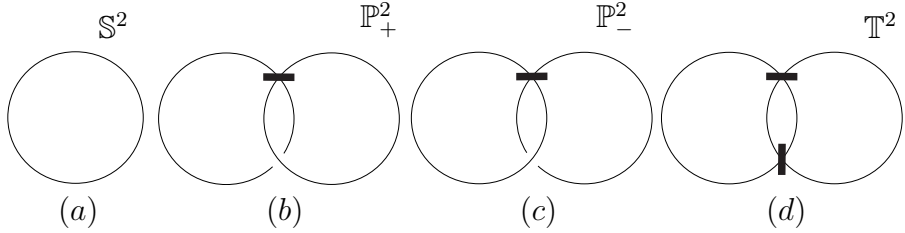

\begin{center}
\begin{lpic}[b(0.7cm),t(0.7cm)]{./PICTURES/m005(11.9cm)}
  \lbl[b]{18,25;$\mathbb{S}^2$}
  \lbl[b]{63,25;$\mathbb{P}^2_+$}
	\lbl[b]{103,25;$\mathbb{P}^2_-$}
  \lbl[b]{148,25;$\mathbb{T}^2$}
  \lbl[t]{12,-2;$(a)$}
  \lbl[t]{47,-2;$(b)$}
	\lbl[t]{90,-2;$(c)$}
  \lbl[t]{132,-2;$(d)$}
	\end{lpic}
\caption{Examples of the unknotted surfaces.\label{m005}}
\end{center}
\end{figure}

\begin{definition}
A non-split ch-diagram (on a sphere) is \emph{hard} if one cannot make on it any Yoshikawa move that doesn't increase the number of crossings (classical or marked) and is not the standard surface-unlink diagram.
\end{definition}

\begin{example}\label{exe1}
An EPD code for the marked graph diagram presented in Fig.\;\ref{M_1_C_10_K_1} is 
$X[1, 5, 2, 4], X[18, 10, 19, 1], Y[5, 19, 6, 20], X[14, 2, 15, 3],$\\ 
$X[3, 13, 4, 14], X[17, 12, 18, 13], $$X[9, 6, 10, 7], X[20, 16, 17, 15],$\\ 
$X[7, 12, 8, 11], X[16, 9, 11, 8]$, it turns out to be the minimal hard prime diagram of $\mathbb{S}^2\sqcup\mathbb{P}^2$.
\end{example}

\begin{figure}[ht]
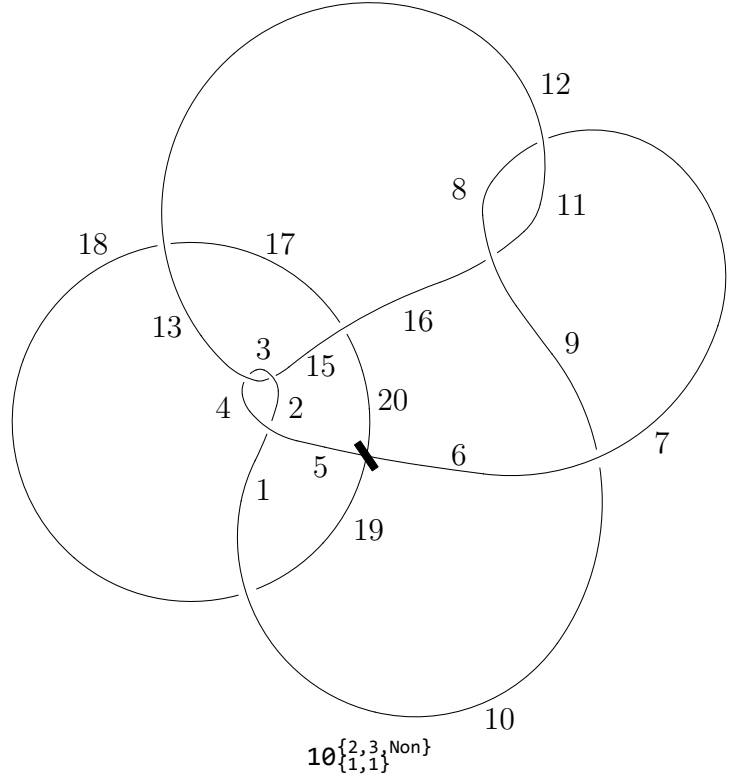

\begin{center}
\begin{lpic}[]{./PICTURES/M_1_C_10_K_1(9.5cm)}
  \lbl[l]{30,35;1}
	 \lbl[l]{34,45;2}
	 \lbl[b]{31,51;3}
	 \lbl[r]{27,45;4}
	  \lbl[t]{38,39;5}
	 \lbl[b]{55,38;6}
	 \lbl[t]{80,42;7}
	 \lbl[r]{56,72;8}
	  \lbl[l]{68,53;9}
	 \lbl[t]{60,8;10}
	  \lbl[l]{67,70;11}
	 \lbl[l]{65,85;12}
	 \lbl[r]{21,55;13}
		  \lbl[t]{38,51.5;15}
	  \lbl[t]{50,57;16}
			 \lbl[b]{33,64;17}
	 \lbl[b]{10,64;18}
		  \lbl[l]{42,30;19}
	  \lbl[l]{45,46;20}
	\end{lpic}
\caption{A hard prime marked diagram.\label{M_1_C_10_K_1}}
\end{center}
\end{figure}

Notice that we have already generated the set $S_4$ of diagrams that don't have opportunity to make on it any of the reducing $\Omega_1$, $\Omega_6$, $\Omega_6'$ type moves, so we proceed to translate the other Yoshikawa generating moves.

\noindent{\bf Zero markers case.}

The Yoshikawa moves without marked vertices are the moves of type $\Omega_1, \Omega_2, \Omega_3$, i.e. the standard Reidemeister moves.

\begin{proposition}
The opportunity to make $\Omega_2$ type move in a diagram from the set $S_4$ is equivalent to its EPD code having the following property. There exists a pair of distinct code-edges such that they are both in exactly two different code-crossings and their positions in code-crossings are ($\{2,2\}$ and $\{1,3\}$) or ($\{1,3\}$ and $\{4,4\}$) and the code-crossing names for that two crossings are both $X$.
\end{proposition}

\begin{proposition}
The opportunity to make $\Omega_3$ type move in a diagram from the set $S_4$ is equivalent to its EPD code having the following property. There exists a triple of distinct code-edges such that they form a triangle on a sphere and their positions in the code-crossings (the vertices of the triangle) are ($\{1,3\}$, $\{4,4\}$ and $\{1,4\}$) or ($\{1,3\}$, $\{2,4\}$ and $\{3,4\}$) or ($\{1,3\}$, $\{2,2\}$ and $\{2,3\}$) or ($\{1,3\}$, $\{2,4\}$ and $\{1,2\}$) or ($\{1,3\}$, $\{4,4\}$ and $\{3,4\}$) or ($\{1,3\}$, $\{2,4\}$ and $\{2,3\}$) or ($\{1,3\}$, $\{2,2\}$ and $\{1,2\}$) or ($\{1,3\}$, $\{2,4\}$ and $\{1,4\}$) and the code-crossing names for that three crossings are all $X$.
\end{proposition}

Other interpretation of the relationship between Reidemeister moves and PD codes one can find in \cite{Mat15}. Consider the case of ch-diagrams with zero markers, it is the case of trivial $2$-links because any other base surface beside $\mathbb{S}^2$ must have at least one marked vertex in any of its ch-diagram and that follows from the fact that each of its connected components satisfies the inequality: $\text{the number of marked vertices}\geq 2-\text{Euler characteristic}$ (see \cite{Jab16}).

This family is equivalent to the classical case of unlinked circles in the $3$-space (with generic diagrams on a $2$-sphere). Therefore, in this case we send the reader to the other papers concerning this topic, for example \cite{KauLam11}. Our computational results extend the known results for the classical case of minimal hard prime diagrams as follows.

\begin{theorem}
Up to mirror image, the only minimal hard prime classical unlink diagram with two components is $8^{\{2,4,Ori\}}_{\{0,1\}}$, and the only minimal hard prime classical unlink diagram with three components is $12^{\{3,6,Ori\}}_{\{0,424\}}$ (see Fig.\;\ref{hard08}). Furthermore, the only minimal hard prime classical unknot diagrams are the four diagrams shown in Fig.\;\ref{hard09}.
\end{theorem}

\begin{figure}[ht]
\includegraphics[width=0.495\textwidth]{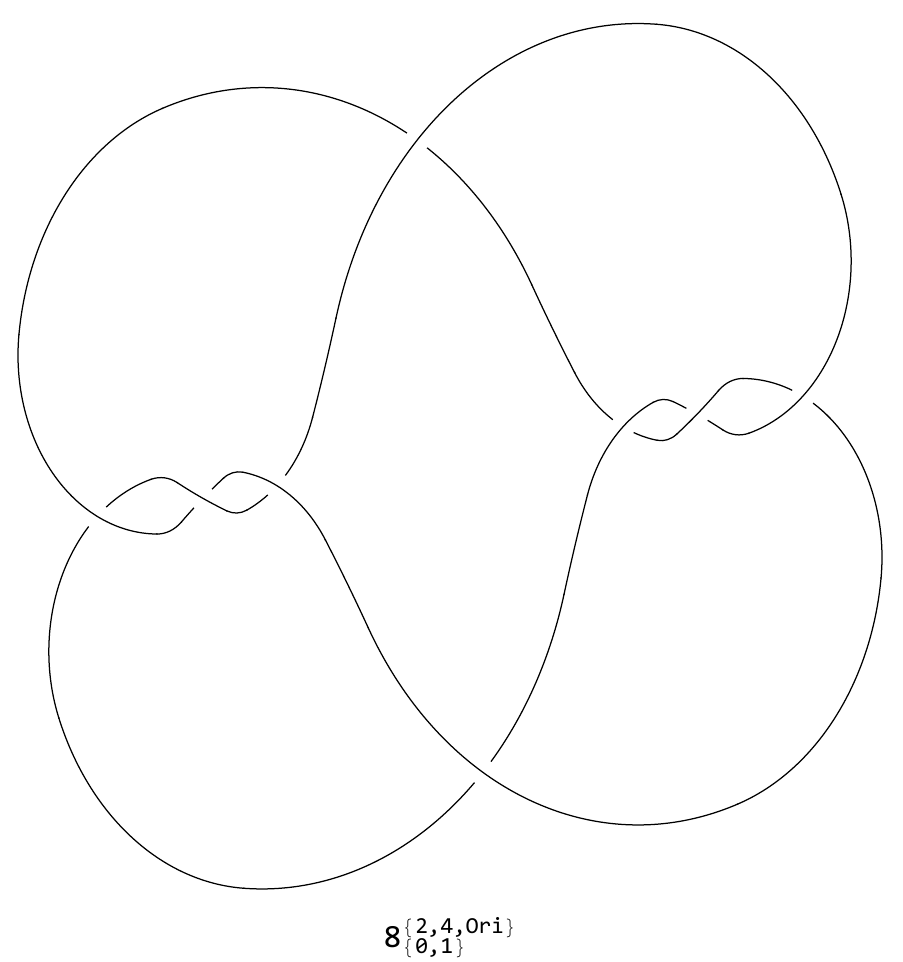}
\includegraphics[width=0.495\textwidth]{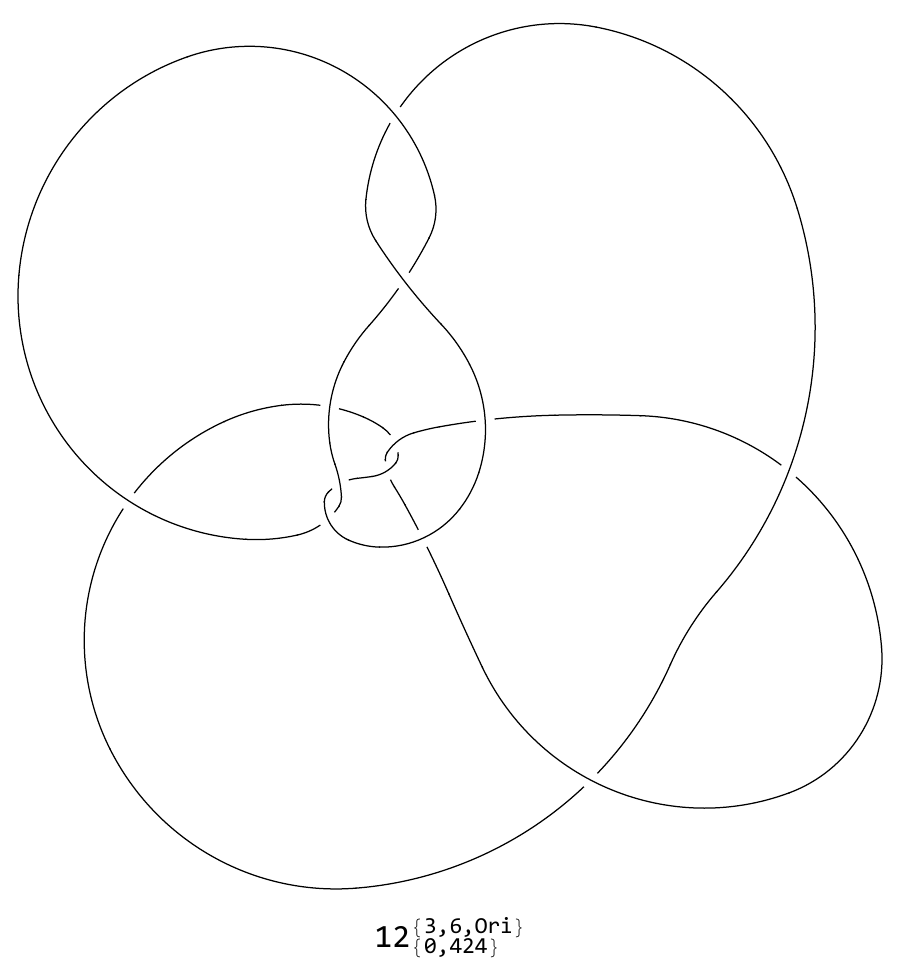}
\caption{The minimal hard prime two and three component unlink diagram.\label{hard08}}
\end{figure}

\begin{figure}[ht]
\includegraphics[width=0.495\textwidth]{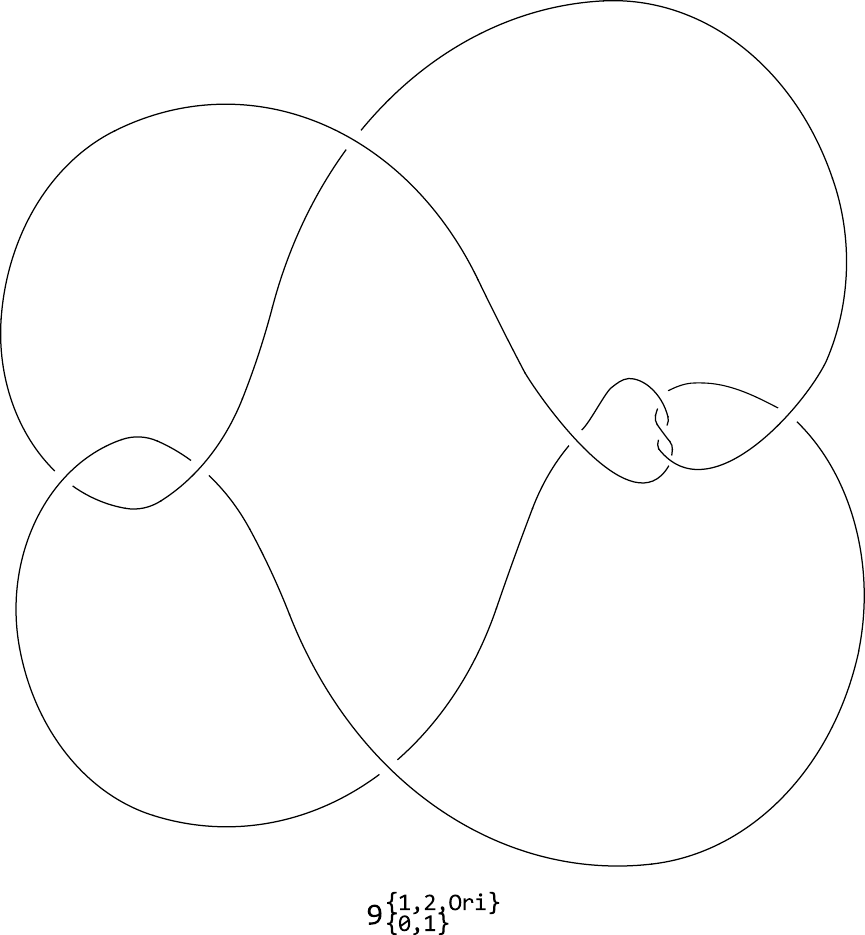}
\includegraphics[width=0.495\textwidth]{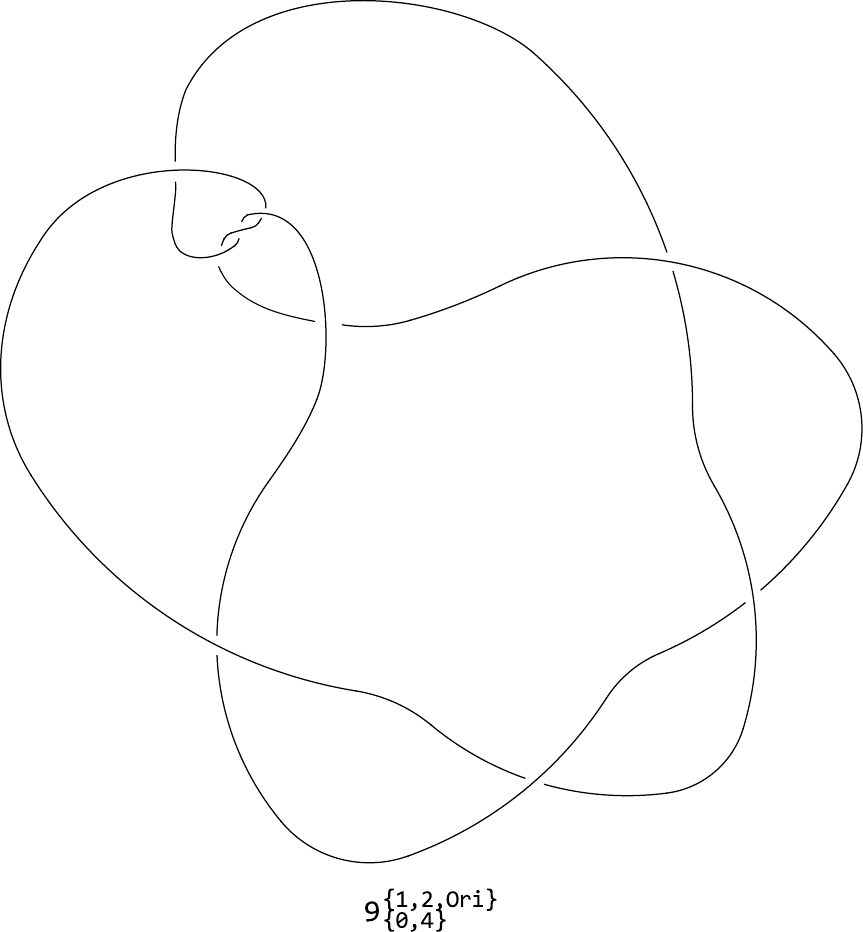}
\includegraphics[width=0.495\textwidth]{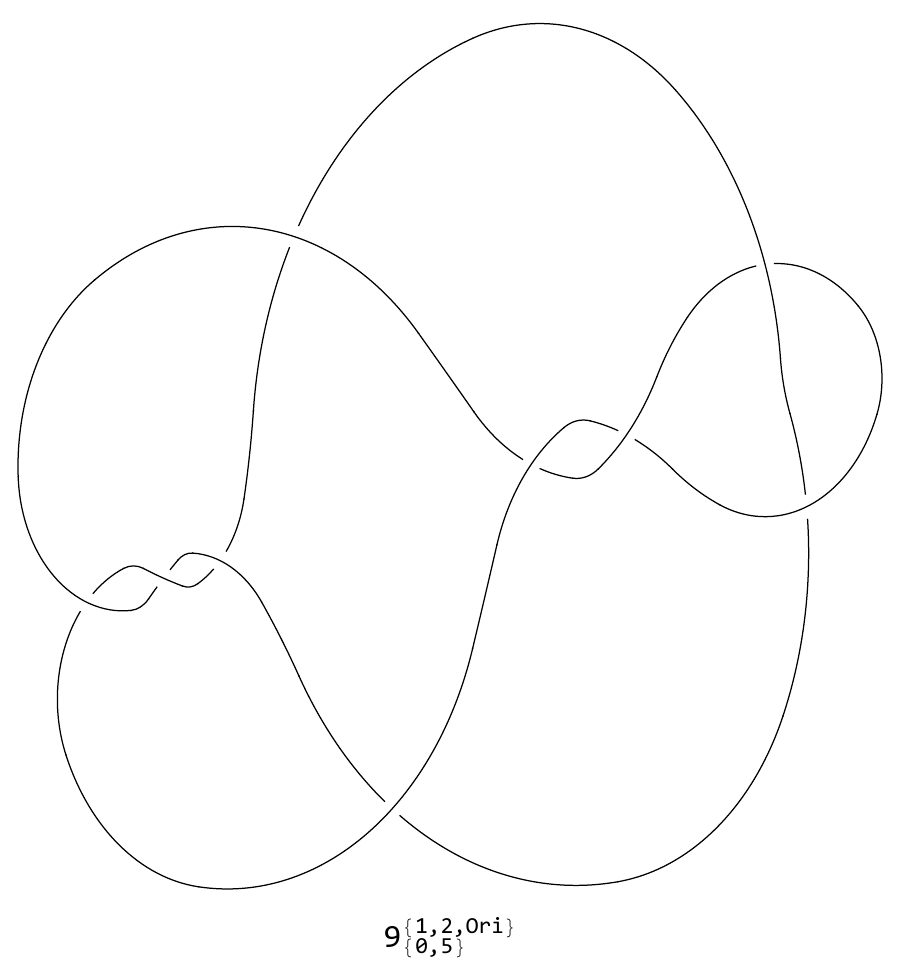}
\includegraphics[width=0.495\textwidth]{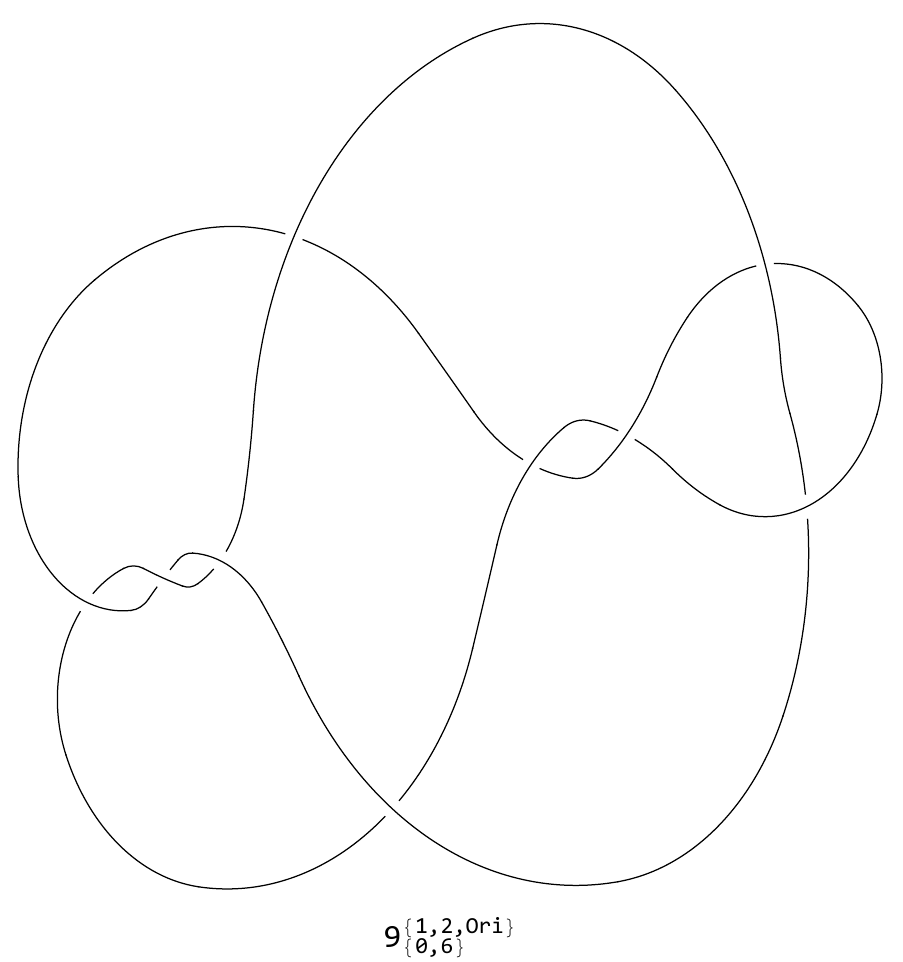}
\caption{Minimal hard prime unknot diagrams.\label{hard09}}
\end{figure}

\begin{proof}
The computation results by the described method produced hard classical diagrams (with $n$ crossings) in the amounts presented in column $M_0$ of Table\;\ref{table3}. We then identify which two among them are related by spherical isotopy or mirror reflection. There were two cases to consider for $8$-crossing diagrams, six diagrams for the unknot for $9$-crossing diagrams and only one diagram for three component unlink. One can see that the diagrams $9^{\{1,2,Ori\}}_{\{0,5\}}$ and $9^{\{1,2,Ori\}}_{\{0,6\}}$ cannot be transformed one to the other neither by a spherical isotopy because they have only one $5$-gon (so we cannot transform this region to some other) nor by a mirror reflection because they both have the same nonzero writhe. The diagrams $9^{\{1,2,Ori\}}_{\{0,1\}}$, $9^{\{1,2,Ori\}}_{\{0,4\}}$ and $9^{\{1,2,Ori\}}_{\{0,5\}}$ are pairwise distinct because the first one has zero $5$-gons and the second has two $5$-gons.
\end{proof}

\noindent{\bf One marker case.}

The Yoshikawa moves with exactly one marked vertex involved are the moves of type $\Omega_4, \Omega_4', \Omega_5, \Omega_6, \Omega_6'$.

\begin{proposition}
The opportunity to make $\Omega_4$ or $\Omega_4'$ type move in a diagram from the set $S_4$ is equivalent to its EPD code having the following property. There exists a triple of distinct code-edges such that they form a triangle on a sphere and exactly one of the code-crossing names for that three crossings (vertices of that triangle) are not $X$ and their positions in the code-crossings, starting from two edges that meet in the vertex named $X$ are ($\{1,3\}$, $\{4,4\}$ and $\{1,4\}$) or ($\{1,3\}$, $\{2,4\}$ and $\{3,4\}$) or ($\{1,3\}$, $\{2,2\}$ and $\{2,3\}$) or ($\{1,3\}$, $\{2,4\}$ and $\{1,2\}$) or ($\{1,3\}$, $\{4,4\}$ and $\{3,4\}$) or ($\{1,3\}$, $\{2,4\}$ and $\{2,3\}$) or ($\{1,3\}$, $\{2,2\}$ and $\{1,2\}$) or ($\{1,3\}$, $\{2,4\}$ and $\{1,4\}$).
\end{proposition}

\begin{proposition}
The opportunity to make $\Omega_5$ type move in a diagram from the set $S_4$ is equivalent to its EPD code having the following property. There exists a pair of distinct code-edges such that they are both in exactly two different code-crossings and exactly one of the code-crossing names from that two crossings are $X$.
\end{proposition}

In the case of ch-diagrams with exactly one marker we have surface-links that must have the base surface components being either $\mathbb{S}^2$ or $\mathbb{P}^2$ and they are all surface-unlinks \cite{Yos94}.

\begin{table}[ht]
\caption{Numerical results of our computation.\label{table3}}
{\begin{tabular}{@{}l|lllllllllll@{}}
	$n$&$M_0$&$M_1$&$M_2$&$M_3$&$M_4$&$M_5$&$M_6$&$M_7$&$M_8$&$M_9$&$M_{10}$\\
	\hline
		3&0&0&0&0&-&-&-&-&-&-&-\\
		4&0&0&0&0&1&-&-&-&-&-&-\\
		5&0&0&0&0&1&0&-&-&-&-&-\\
		6&0&0&0&0&8&0&2&-&-&-&-\\
		7&0&0&2&0&6&3&5&0&-&-&-\\
		8&2&0&7&4&55&11&53&2&6&-&-\\
		9&7&0&39&9&138&103&214&43&32&0&-\\
		10&30&2&221&125&859&591&1463&658&533&14&17\\
		11&107&&&&&&&&5104&675&229\\
		12&424&&&&&&&&&&\\
\end{tabular}}
\end{table}

\noindent{\bf At least two markers case.}

The Yoshikawa moves with at least two marked vertices involved are the moves $\Omega_7$ and $\Omega_8$.

\begin{proposition}
The opportunity to make $\Omega_7$ type move in a diagram from the set $S_4$ is equivalent to its EPD code having the following property. There exists a code-edge such that the code-crossing names for that crossing (the endings of the code-edge) are (both $Y$ or $Z$ and their positions in code-crossings have different parity) or (are $Y$ and $Z$ and their positions in code-crossings have the same parity), see Fig.\;\ref{pic084}.
\end{proposition}

\begin{figure}[ht]
\begin{center}
\begin{lpic}[b(1cm),t(1cm)]{./PICTURES/om_7_opp(7cm)}
  \lbl[t]{10,21;a}
	\lbl[t]{47,20;a}
	\lbl[l]{1,45;Z[*,*,a,*]=Z[a,*,*,*]=Y[*,*,*,a]=Y[*,a,*,*]}
	\lbl[l]{1,-5;Y[*,*,a,*]=Y[a,*,*,*]=Z[*,*,*,a]=Z[*,a,*,*]}
	\end{lpic}
		\caption{The opportunity to make $\Omega_7$ type move\label{pic084}}
\end{center}
\end{figure}

Finally one can see in Fig.\;\ref{pic002} that the opportunity to make $\Omega_8$ type move in a diagram from the set $S_4$ implies the opportunity to make $\Omega_4$ or $\Omega_4'$ type move.

Taking all of the above propositions (which can be checked by considering all possible combinatorial configurations) and described method of generating valid EPD codes of diagrams, we are now ready to generate and enumerate hard prime ch-diagrams.

Our notation in the form of $A^{\{B,C,D\}}_{\{E,F\}}$ for a marked graph diagram $M$ of a surface link $L$ means that $M$ has $A$ crossings, $E$ marked crossings, $F$ as the index of the diagram generated within the family of diagrams with the same pair $(A, E)$, the surface $L$ has $B$ number of components, Euler characteristic $C$ and the orientability $D\in\{Ori, Non\}$.

The numerical summary of the results from our described algorithm for searching for hard prime ch-diagrams are presented in Table\;\ref{table3}, where the number in column $M_k$ and row $n$ is the number of obtained diagrams with $n-k$ classical crossings and $k$ singular crossings.

\begin{remark}
In our enumeration some diagrams are the same after spherical isotopy or taking mirror images, for example the triple $6^{\{1,0, Non\}}_{\{4,6\}}$, $6^{\{1,0, Non\}}_{\{4,7\}}$, $6^{\{1,0, Non\}}_{\{4,8\}}$ have the same shadow that is periodic and after taking mirror images (with respect to the classical and marked vertices) and planar rotation they represent the identical diagram.
\end{remark}

We can check the orientability of the base surface for a surface-link represented by a connected marked graph diagram using the following method. Starting from orienting the first (non-labeled) edge, label code-edge in two code-crossings it belongs to by $T$ when it is oriented "toward" the crossing and $F$ if it is oriented "from" the crossing. Then, label all opposing code-edges in that two code-crossings $F$ or $T$ with the rule that when the code name for the crossing is $X$ then the opposing edge takes different label and otherwise the opposing edge takes the same label (in that crossing) see the convention on the left of Fig.\;\ref{pic004}.

Then, we label the non-labeled edges in the two newly considered vertices (endpoints of the already oriented edge) the appropriate letters to meet the crossing convention (when its name is $X$ and it meets new surface-component we choose one of the two possibilities). We repeat the process if there are non-labeled code-edges in visited vertices. We end up with the oriented graph with the orientation we will call here an \emph{abstract orientation} (an example is on the right of Fig.\;\ref{pic004}). One can easily verify the following.

\begin{proposition}\label{prop2}
The orientability of the surface-link represented by an abstractly oriented diagram from the set $S_4$ is equivalent to its EPD code having the following property. For all code-crossings named $X$ the code-edges are labeled by $[T,F,F,T]$ or $[F,F,T,T]$ or $[T,T,F,F]$ or $[F,T,T,F]$ and for all code-crossings named $Y$ or $Z$ the code-edges are labeled by $[T,F,T,F]$ or $[F,T,F,T]$.
\end{proposition}

\begin{center}
\begin{figure}[ht]
\begin{center}
\begin{center}
\includegraphics[width=0.23\textwidth]{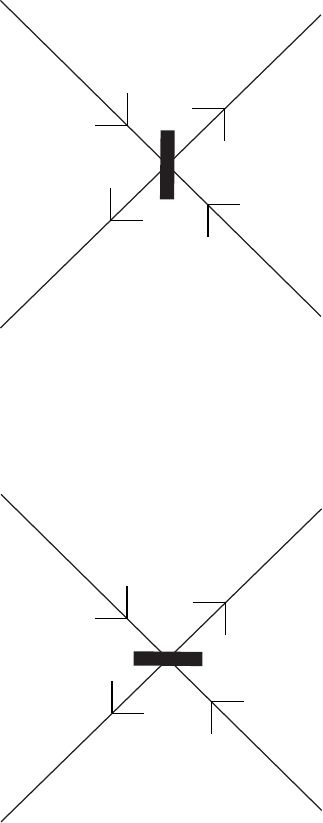}
\includegraphics[width=0.53\textwidth]{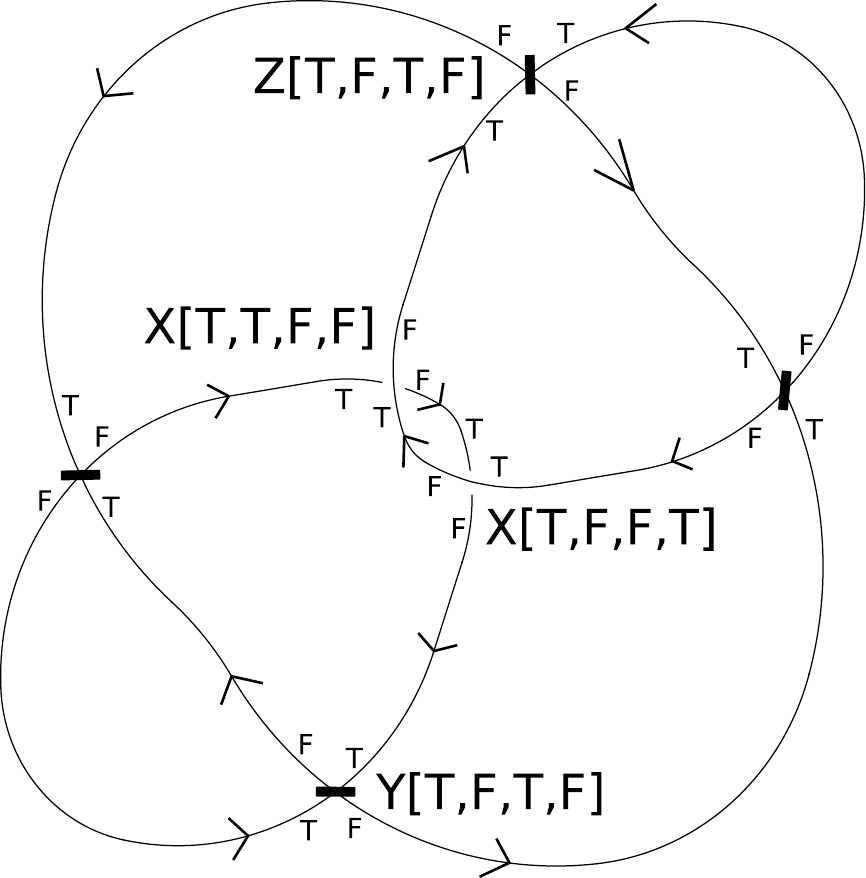}
\end{center}
\end{center}
		\caption{An orientation near a marked vertex, and an example\label{pic004}}
\end{figure}
\end{center}

The results presented in Table\;\ref{table3} lead us to the following.

\begin{theorem}
There are no hard marked graph diagrams without classical crossings having an odd number of (marked) crossings.
\end{theorem}

\begin{proof}
Assume the contrary. This diagram is a $4$-regular graph with an odd number of vertices therefore it cannot be bipartite. By Konig Theorem \cite{Kon01} it has then an odd cycle, so it either makes the opportunity to make $\Omega_6$ or $\Omega_6'$ type move or (because every vertex is $4$-valent) some (connected by an edge) two vertices from the cycle have the types of markers that mark the same region and that makes the opportunity to make $\Omega_7$ type move, which contradicts the hardness of the marked graph diagram.
\end{proof}

\begin{proposition}[\cite{Yos94}]\label{prop1}
Every nontrivial surface-link with its ch-diagram having $c$ crossings, including $m$ marked crossings, must satisfy inequalities $2\leq m\leq c-4$.
\end{proposition}

\begin{theorem}
All minimal hard prime surface-unlink diagrams (with prime base surface components) up to mirror images and up to $10$ crossings are as follows (presented in Fig.\;\ref{M_1_C_10_K_1}, \ref{hard08}, \ref{hard09}, \ref{hardM}).
\begin{itemize}
    \item  $\mathbb{T}^2$ with the diagram $4^{\{1,0, Ori\}}_{\{4,1\}}$.
		\item  $\mathbb{S}^2\sqcup\mathbb{S}^2$ with the diagram $8^{\{2,4,Ori\}}_{\{0,1\}}$.
		\item  $\mathbb{S}^2$ with the diagrams $9^{\{1,2,Ori\}}_{\{0,1\}}$, $9^{\{1,2,Ori\}}_{\{0,4\}}$, $9^{\{1,2,Ori\}}_{\{0,5\}}$, $9^{\{1,2,Ori\}}_{\{0,6\}}$, $9^{\{1,2,Ori\}}_{\{2,38\}}$.
		\item  $\mathbb{P}^2$ with the diagram $9^{\{1,1,Non\}}_{\{2,39\}}$.
		\item  $\mathbb{S}^2\sqcup\mathbb{P}^2$ with the diagram $10^{\{2,3,Non\}}_{\{1,1\}}$.
\end{itemize}
\end{theorem}

\begin{proof}
We partition further the computational results described above by the topological type of the base surface (using Prop.\;\ref{prop2} on each diagram component, Euler characteristic and the number of components) and by the triviality of the surface-link type (using Prop.\;\ref{prop1}, the first Alexander ideal and the standard surface-link table \cite{Yos94}, where errors in the invariant calculations were corrected in \cite{INO18}).
\end{proof}

\begin{figure}[ht]
\includegraphics[width=0.495\textwidth]{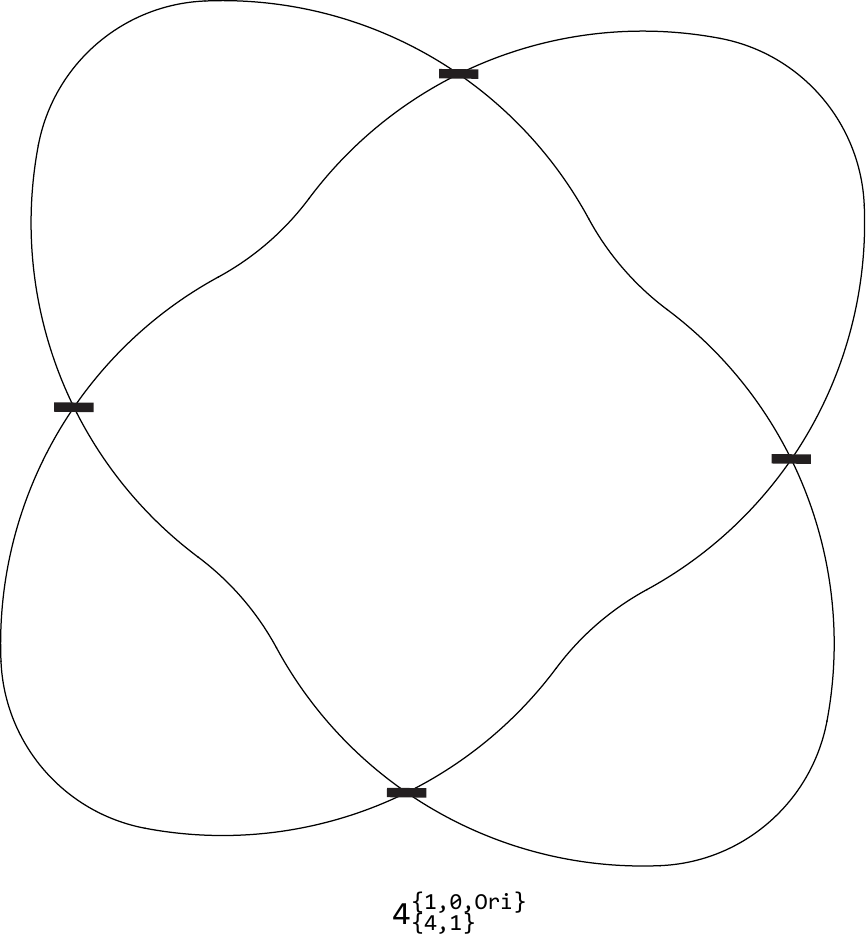}
\includegraphics[width=0.495\textwidth]{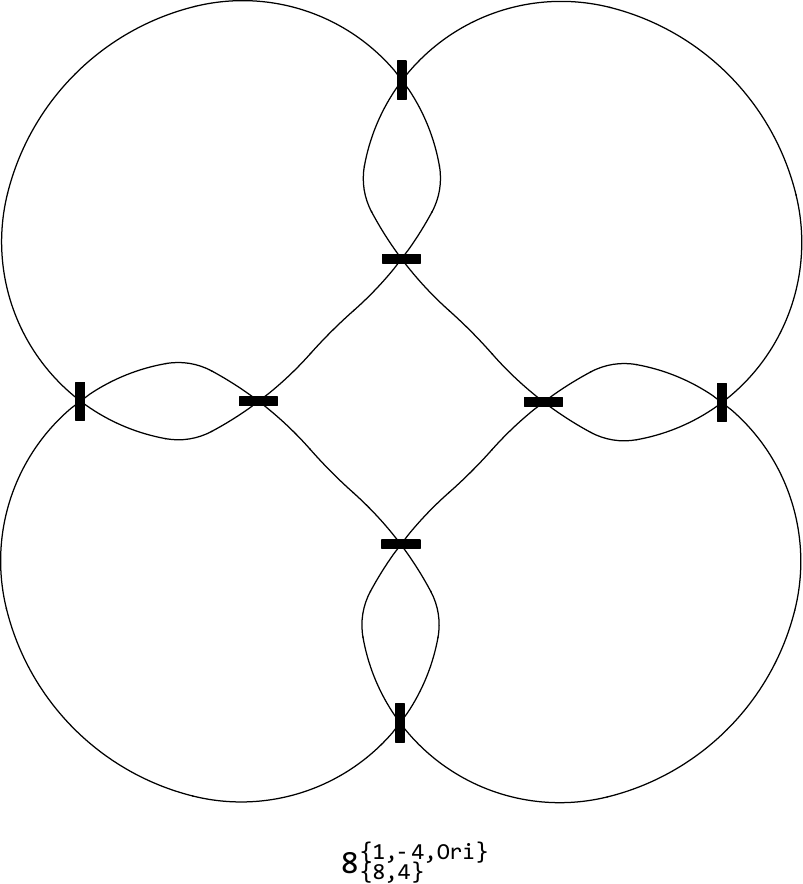}
\includegraphics[width=0.495\textwidth]{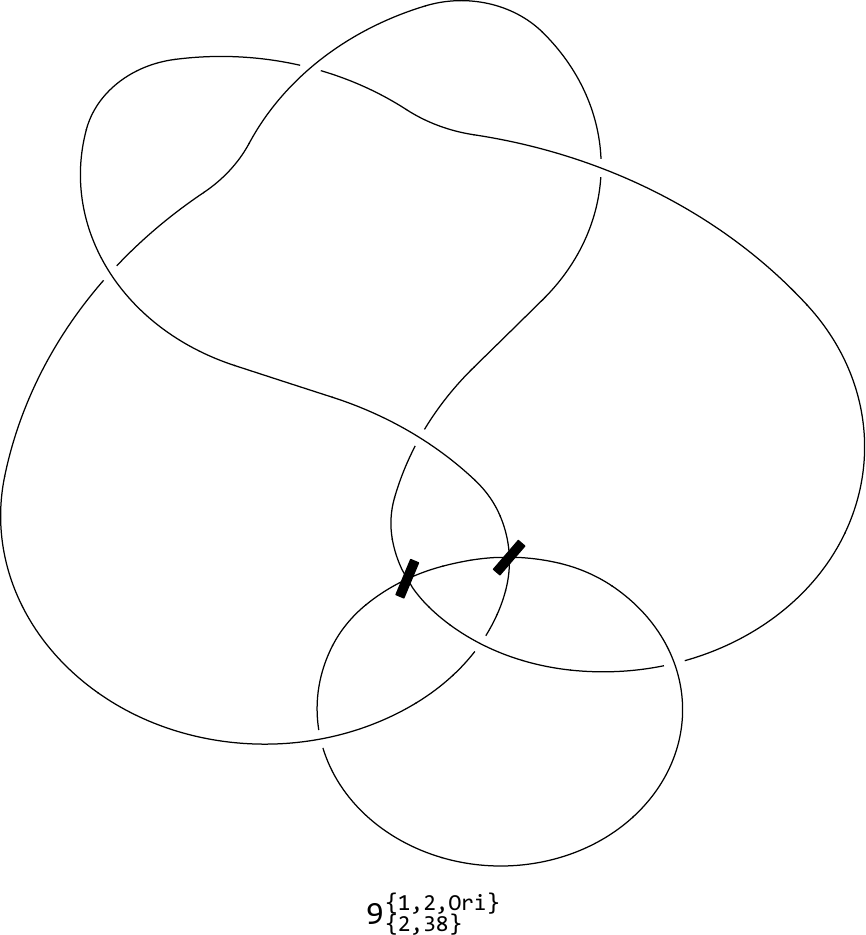}
\includegraphics[width=0.495\textwidth]{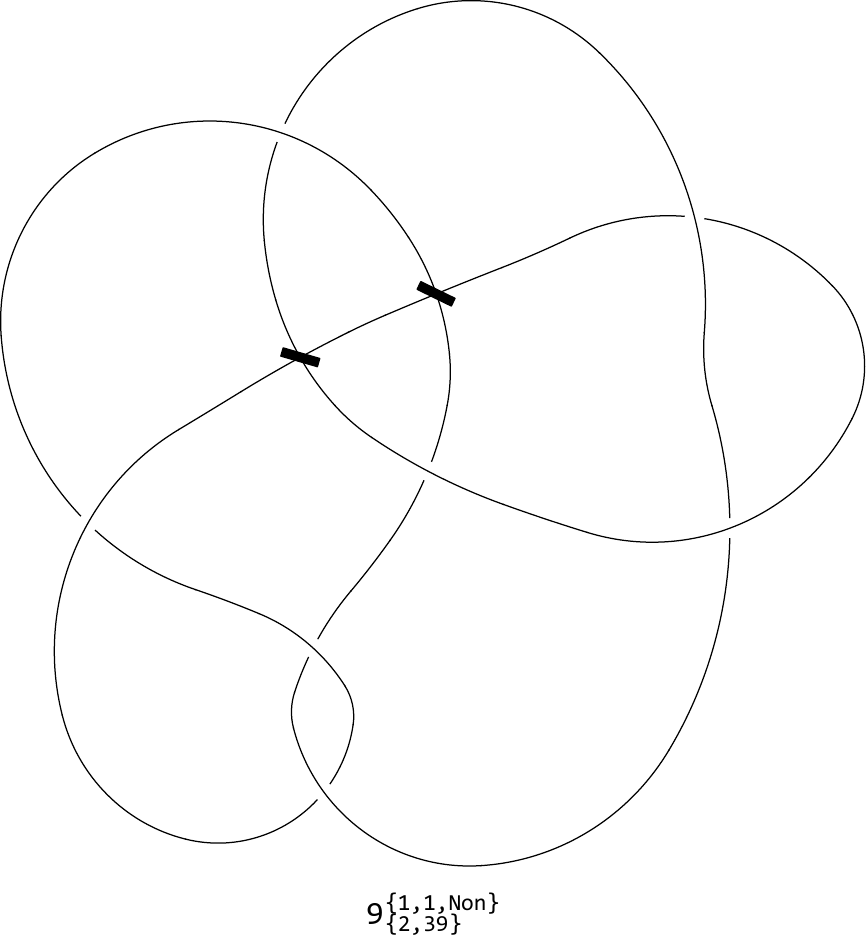}
\caption{Minimal hard prime surface-unlink diagrams with at least one marked vertex.\label{hardM}}
\end{figure}

\begin{remark}
As the computations show, there is the case of hard prime marked ch-diagram for $2$-knots with less than $9$ crossings namely $8^{\{1,2,Ori\}}_{\{2,7\}}$ but it is not the $2$-unknot because its first elementary ideal is generated by the Alexander polynomial of the trefoil, therefore it is the spun $2$-knot of the trefoil. It is, therefore, a kind of surprise that this hard diagram has the same number of classical and marked vertices as its standard diagram $8_1$ (see \cite{JKL15}, \cite{KeaKur08}, \cite{Lom81}, \cite{Yos94}) which is not hard. 

The same property (up to mirror image) shares for example the diagram $9^{\{1,2,Ori\}}_{\{2,31\}}$ which is $9_1$ in standard tables, the diagram $9^{\{2,2,Ori\}}_{\{2,21\}}$ which is $9_1^{0,1}$, the diagram $9^{\{2,0,Non\}}_{\{4,79\}}$ which is $9_1^{1,-2}$, the diagram $10^{\{1,2,Ori\}}_{\{2,121\}}$ which is $10_{1}$, the diagram $10^{\{2,2,Ori\}}_{\{2,163\}}$ which is $10_1^{0,1}$, the diagram $10^{\{2,2,Non\}}_{\{2,68\}}$ which is $10_1^{0,-2}$ and the diagram $10^{\{2,0,Non\}}_{\{4,426\}}$ which is $10_1^{-2,-2}$. All diagrams from this remark are shown in Fig.\;\ref{hardNT}-\ref{hardNT2}, the other diagrams from Yoshikawa's table are either unknotted, hard or their hard diagrams (if they exist) have greater the number of classical or marked crossings.
\end{remark}

\begin{figure}[ht]
\includegraphics[width=0.495\textwidth]{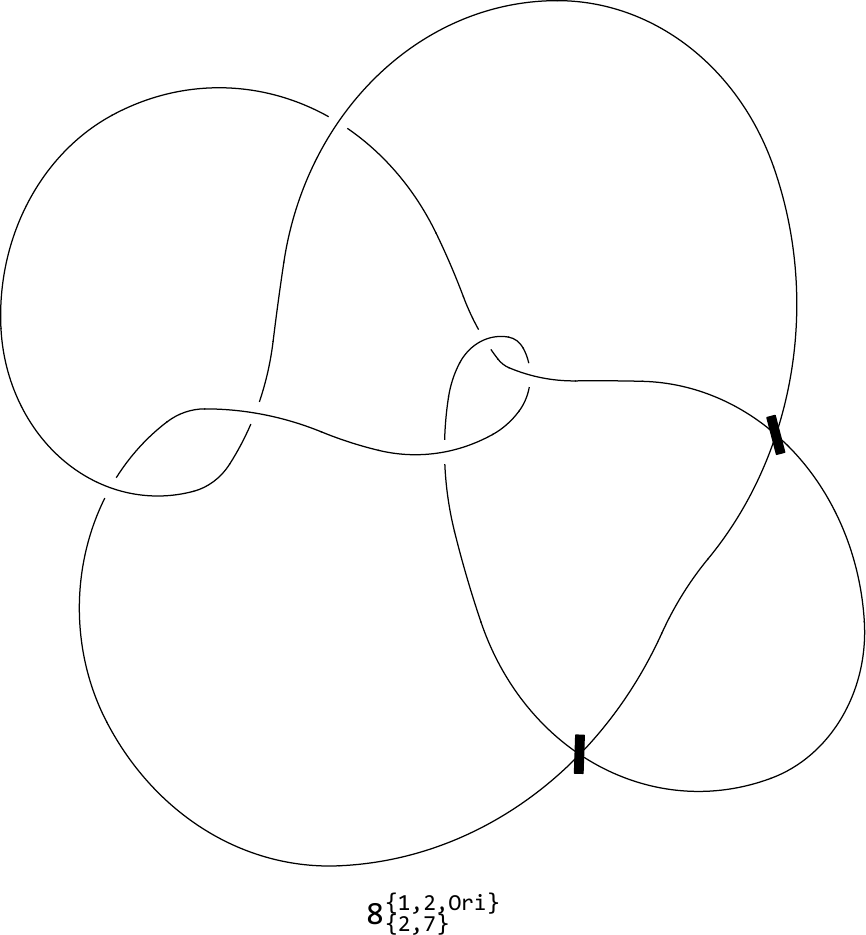}
\includegraphics[width=0.495\textwidth]{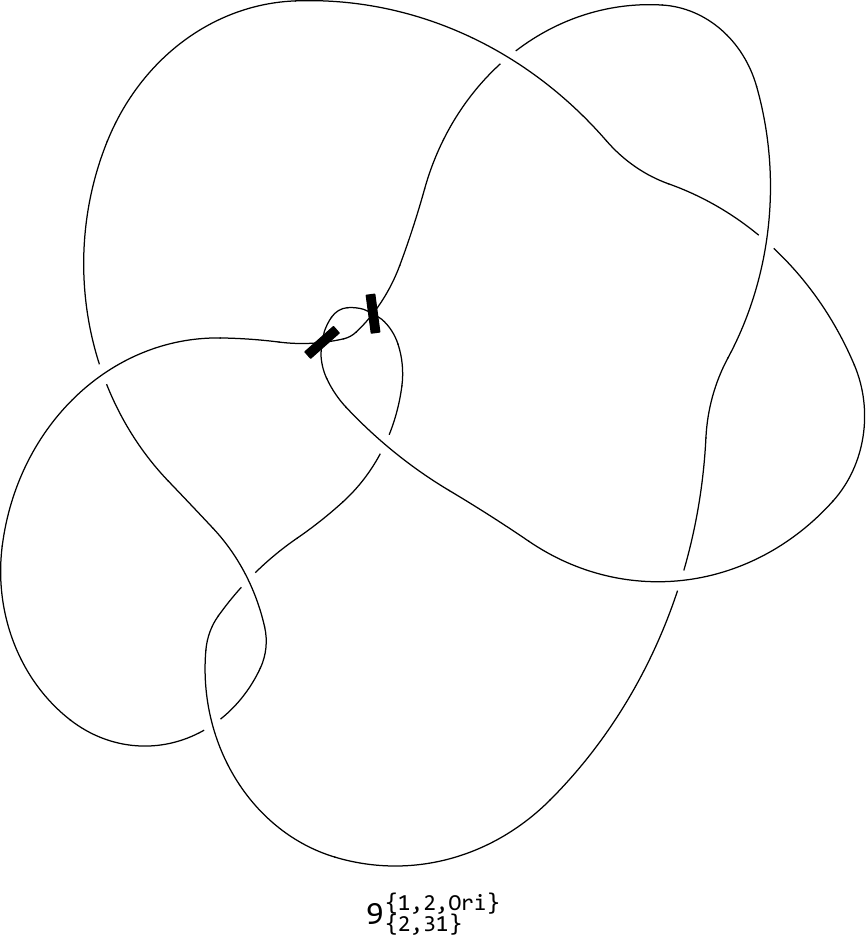}
\includegraphics[width=0.495\textwidth]{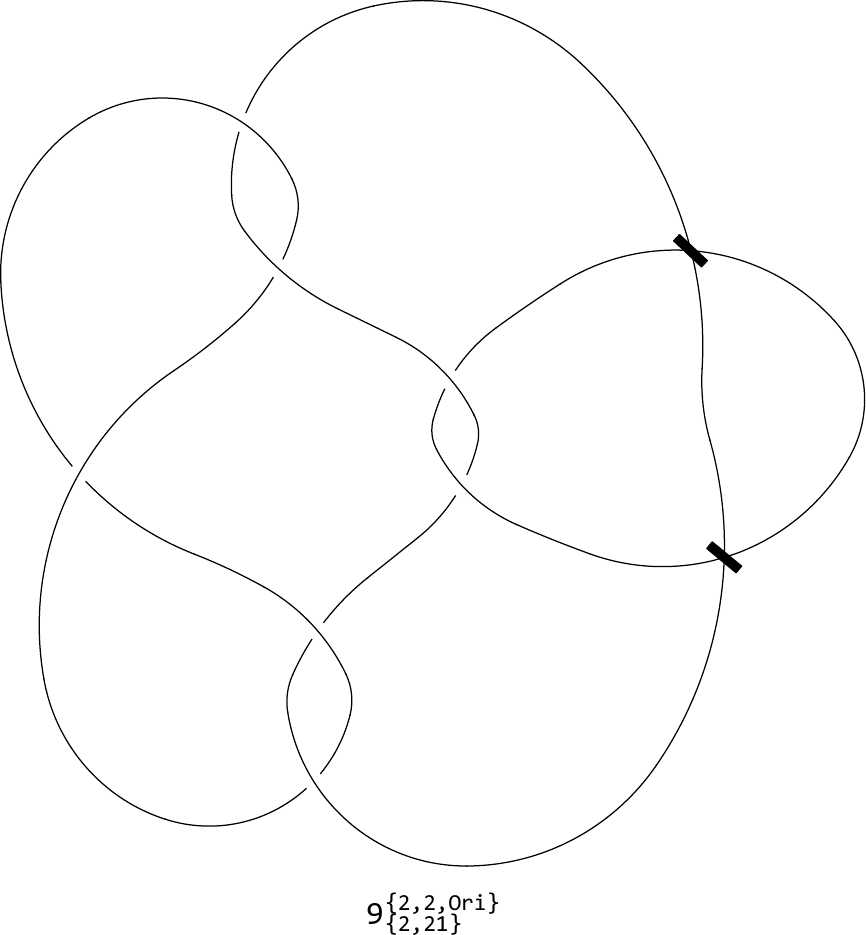}
\includegraphics[width=0.495\textwidth]{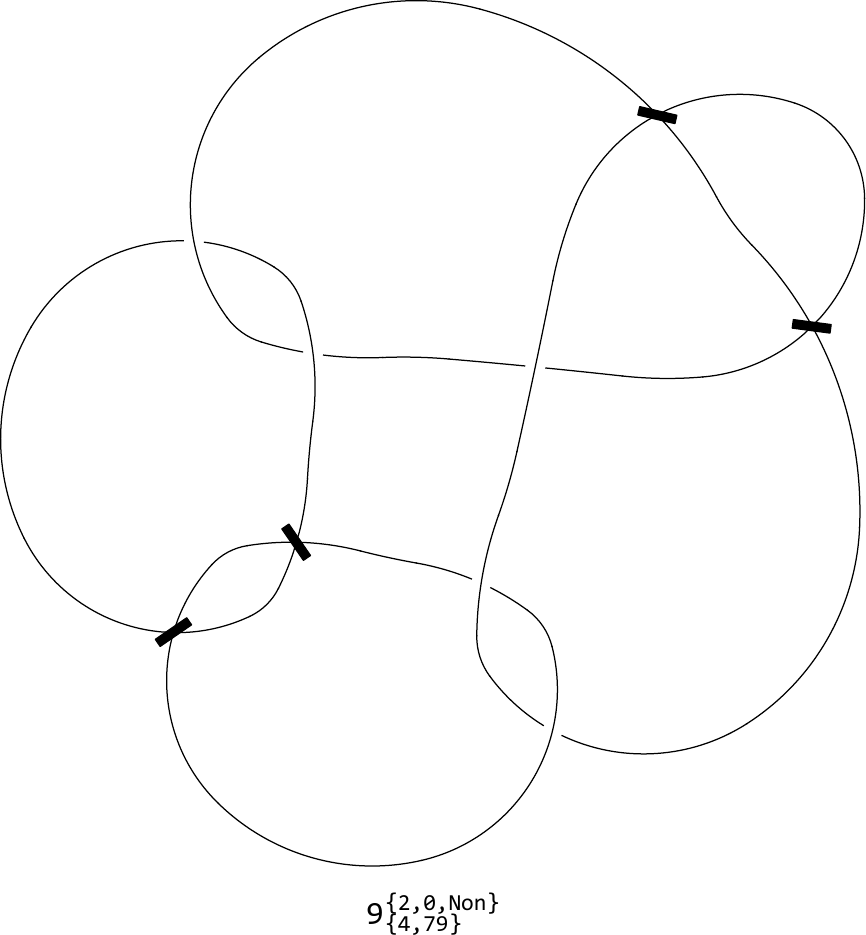}
\caption{Hard nontrivial surface-link diagrams.\label{hardNT}}
\end{figure}

\begin{remark}
The diagrams $4^{\{1,0, Ori\}}_{\{4,1\}}$ and $8^{\{1,-4, Ori\}}_{\{8,4\}}$ in Fig.\;\ref{hardM} give us an idea how to construct a hard prime diagram for an orientable surface-unknot with arbitrarily higher odd genus. Just expand the diagram (as shown in Fig.\;\ref{higher_genus}) with the pattern matching quadruples of marked vertices to form a diagram with appropriate periodicity with is clearly prime and hard. One can transform the diagram to the standard connected sum of unknotted tori using commutativity of marked vertices in $2$-strand braid diagrams (for details see \cite{Jab13}) and then applying $\Omega_7$ type moves.
\end{remark}

\begin{figure}[ht]
\begin{center}
\begin{lpic}[]{./PICTURES/higher_genus(7.5cm)}
	\end{lpic}
		\caption{A hard prime diagram for an orientable surface-unknot with odd genus.\label{higher_genus}}
\end{center}
\end{figure}

\begin{figure}[ht]
\includegraphics[width=0.495\textwidth]{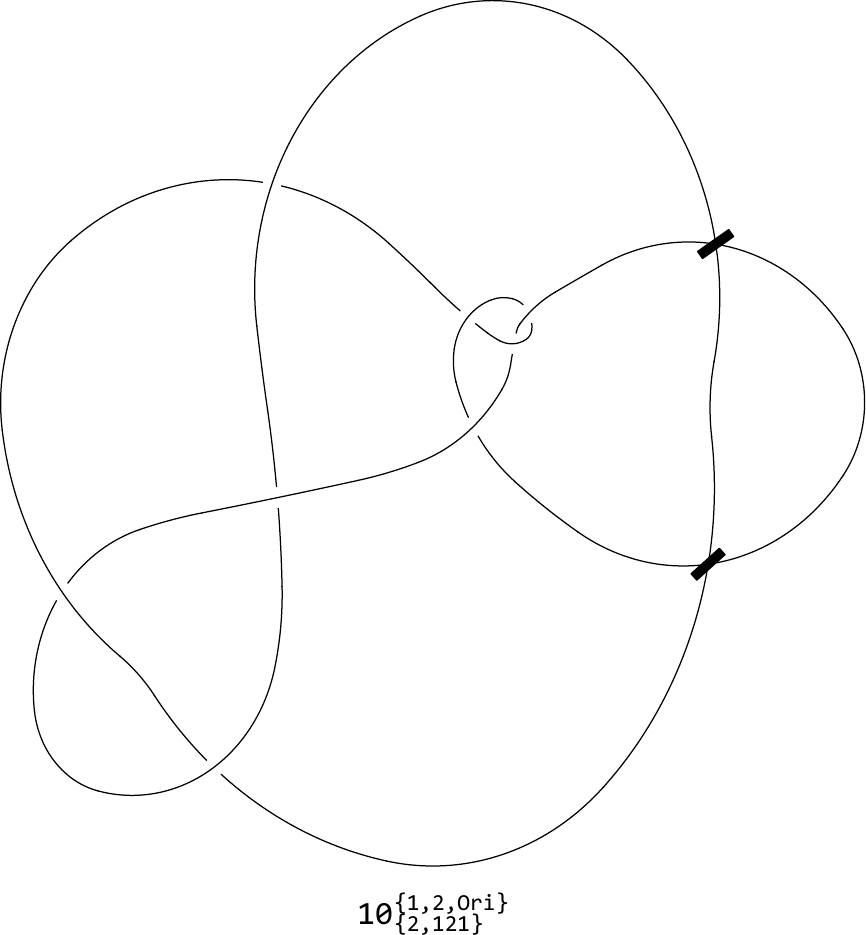}
\includegraphics[width=0.495\textwidth]{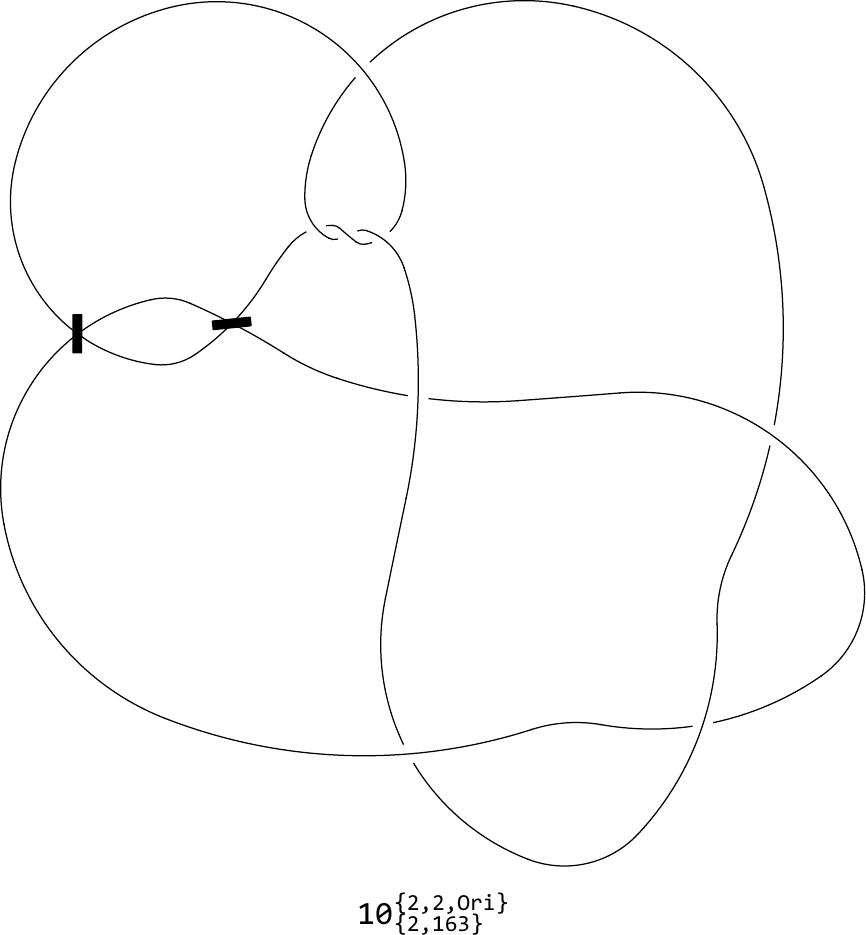}
\includegraphics[width=0.495\textwidth]{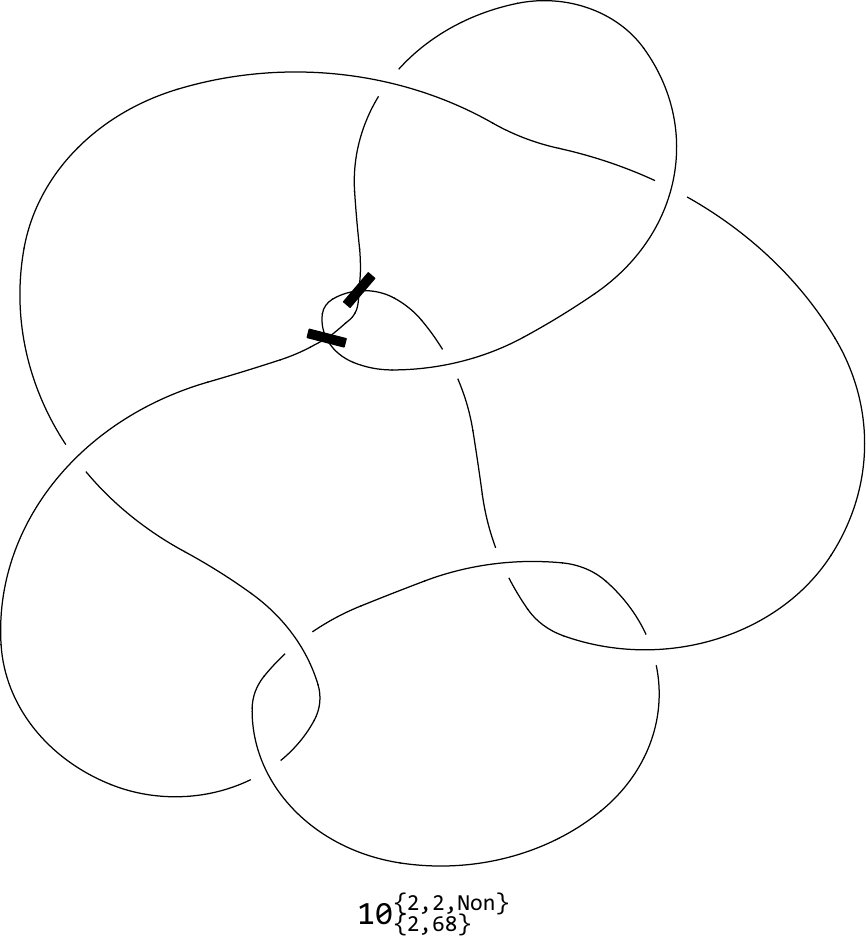}
\includegraphics[width=0.495\textwidth]{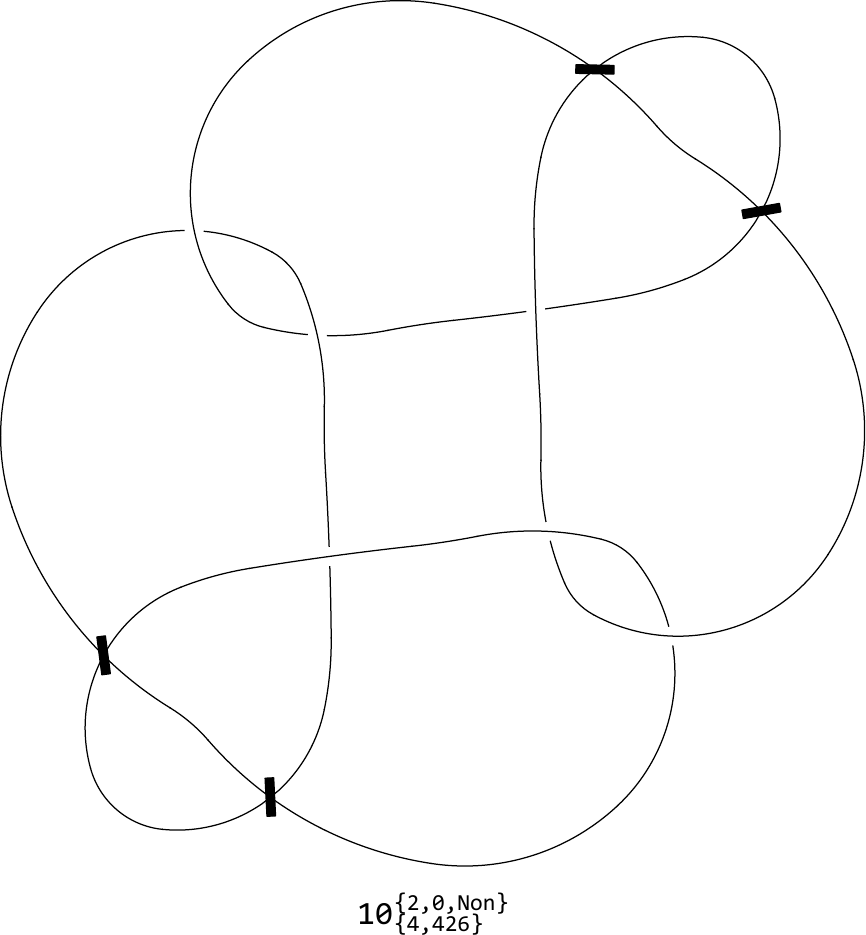}

\caption{Hard nontrivial surface-link diagrams. (cont.)\label{hardNT2}}
\end{figure}

\begin{exercise}
Transform the diagram $9^{\{1,2,Ori\}}_{\{2,38\}}$ in Fig.\;\ref{hardM} to the standard unknot by using Yoshikawa moves.
\end{exercise}

Computational results also naturally led to the following.

\begin{question}
Is every minimal hard prime diagram for the classical unlink (with at least three components) also the minimal hard prime diagram for the surface $2$-unlink?
\end{question}

\begin{question}
What is the number of crossings of minimal hard prime diagrams for surface-unlinks with the base surface $\mathbb{P}^2\sqcup\mathbb{T}^2$ or $\mathbb{T}^2\sqcup\mathbb{T}^2$ or $\mathbb{P}^2\sqcup\mathbb{P}^2$ or $\mathbb{S}^2\sqcup\mathbb{T}^2$? (We know that it has to be at least $11$.)
\end{question}

\end{document}